\pdfoutput=1
\documentclass[a4paper,11pt,oneside]{amsart}
\usepackage[paperheight=279mm,paperwidth=18cm,textheight=26cm,textwidth=14cm,includehead]{geometry}
\usepackage{amsmath,amsfonts,amssymb,amsthm,mathtools}
\usepackage[utf8]{inputenc}
\usepackage[T1]{fontenc}
\usepackage[unicode,bookmarks]{hyperref}
\hypersetup{
colorlinks=true,
citecolor=[rgb]{0,0,0.5},
linkcolor=[rgb]{0,0,0.5},
urlcolor=[rgb]{0,0,0.75},
pdfpagemode=UseNone,
pdfstartview=FitH,
pdfdisplaydoctitle=true,
pdftitle={A double return times theorem},
pdfauthor={Pavel Zorin-Kranich},
pdfsubject={Mathematics Subject Classification (2010): 37A45 (Primary), 26A45 (Secondary)},
pdfkeywords={return times, bilinear ergodic average},
pdflang=en-US
}

\usepackage[style=alphabetic]{biblatex}
\newcommand*{\mailto}[1]{\href{mailto:#1}{#1}}

\numberwithin{equation}{section}
\newtheorem{theorem}[equation]{Theorem}
\newtheorem{lemma}[equation]{Lemma}

\newtheorem{corollary}[equation]{Corollary}
\theoremstyle{definition}
\newtheorem{definition}[equation]{Definition}
\theoremstyle{remark}

\newcommand*{\N}{\mathbb{N}}
\newcommand*{\Z}{\mathbb{Z}}

\newcommand*{\R}{\mathbb{R}}
\newcommand*{\Q}{\mathbb{Q}}

\newcommand*{\dif}{\mathrm{d}}
\newcommand*{\lin}{\mathrm{lin}\,}
\newcommand*{\E}{\mathbb{E}}
\newcommand{\aveN}{\frac{1}{N}\sum_{n=1}^N}
\newcommand{\aveH}{\frac{1}{2H+1}\sum_{h=-H}^H}
\newcommand{\aveHo}{\frac{1}{H}\sum_{h=1}^H} 
\newcommand{\aveFn}{\frac{1}{\abs{\Phi_N}}\sum_{n\in \Phi_N}}

\newcommand{\ld}{\underline{d}}
\newcommand{\one}{\mathbf{1}}

\DeclarePairedDelimiter\abs{\lvert}{\rvert}

\DeclarePairedDelimiter\norm{\lVert}{\rVert}
\DeclarePairedDelimiterX\innerp[2]{\langle}{\rangle}{#1,#2}
\providecommand\given{}
\newcommand\SetSymbol[1][]{%
\nonscript\:#1\vert
\allowbreak
\nonscript\:
\mathopen{}}
\DeclarePairedDelimiterX\Set[1]\{\}{%
\renewcommand\given{\SetSymbol[\delimsize]}
#1
}
\allowdisplaybreaks

\begin{document}
\title{A double return times theorem}
\author{Pavel Zorin-Kranich}
\address{Universit\"at Bonn\\
Mathematisches Institut\\
Endenicher Allee 60\\
53115 Bonn\\
Germany
}
\email{\mailto{pzorin@uni-bonn.de}}
\urladdr{\url{https://www.math.uni-bonn.de/people/pzorin/}}
\begin{abstract}
We prove that for any bounded functions $f_{1},f_{2}$ on a measure-preserving dynamical system $(X,T)$ and any distinct integers $a_{1},a_{2}$, for almost every $x$ the sequence
\[
f_{1}(T^{a_{1}n}x) f_{2}(T^{a_{2}n}x)
\]
is a good weight for the pointwise ergodic theorem.
\end{abstract}
\maketitle

\section{Introduction}\label{sec:intro}
Bourgain's bilinear pointwise ergodic theorem \cite{MR1037434} (see also \cite{MR2417419,MR3623404}) is one of the hardest known convergence results for multiple ergodic averages.
It can therefore be considered surprising that Assani, Duncan, and Moore \cite{MR3492969} have been able to extend it to a Wiener--Wintner type result by relatively simple means, using Bourgain's result as a black box.
We strengthen their result further to a double return times theorem, similarly assuming Bourgain's result as a black box.
\begin{theorem}
\label{thm:double-rtt}
Let $(X,\mu,T)$ be a (not necessarily ergodic) invertible measure-preserving dynamical system and $a_{1},a_{2}$ distinct non-zero integers.
Then for any $f_{1},f_{2}\in L^{\infty}(X)$ there exists a full measure subset $X'\subset X$ such that for every $x\in X'$ the sequence
\[
c_{n} = f_{1}(T^{a_{1}n}x) f_{2}(T^{a_{2}n}x)
\]
is a \emph{good weight for the pointwise ergodic theorem} in the sense that for every further measure-preserving dynamical system $(Y,\nu,S)$ and every $g\in L^{\infty}(Y)$ the limit
\[
\lim_{N\to\infty} \aveN c_{n} g(S^{n}y)
\]
exists for $\nu$-almost every $y\in Y$.
\end{theorem}
This was previously known for weakly mixing systems $(X,\mu,T)$ \cite[Theorem 2]{MR1751656}.
Convergence in $L^{2}(Y)$ in Theorem~\ref{thm:double-rtt} follows from the result of Assani, Duncan, and Moore \cite{MR3492969}.

Our proof relies on the following description of a class of good weights for pointwise convergence of ergodic averages to zero that is implicit in the Bourgain--Furstenberg--Katznelson--Ornstein orthogonality criterion \cite{MR1557098}.
\begin{theorem}[{\cite[Theorem 4.1]{MR1286798}, see also \cite[Theorem 1.2]{arxiv:1301.1884}}]
\label{thm:bfko}
Let $(c_{n})_{n\in\Z}$ be a bounded sequence.
For $\delta>0$ and $0<L<R<\infty$ define
\[
S_{\delta,L,R}(c) :=
\bigcap_{N=L}^{R} S_{\delta,N}(c),
\quad
S_{\delta,N}(c) :=
\Set{h \given \abs[\big]{ \aveN c(n)c(n+h) } < \delta}.
\]
Suppose
\begin{equation}
\label{eq:cond}
\inf_{\delta>0}
\lim_{L\to\infty}
\inf_{R\geq L}
\ld(S_{\delta,L,R}(c)) = 1,
\end{equation}
where the lower density of a set is defined by $\ld(S):=\liminf_{N\to\infty} \abs{S\cap \Set{1,\dots,N}} / N$.
Then for every measure-preserving system $(Y,S)$ and every $g\in L^{\infty}(Y,S)$ we have
\[
\lim_{N\to\infty} \aveN c_{n} g(S^{n}y) = 0
\]
pointwise almost everywhere.
\end{theorem}

The main technical result is that nilfactors of order $2$ are characteristic for the double return times theorem in the following sense.
\begin{theorem}
\label{thm:bilinear-weight-bfko}
Let $(X,T,\mu)$ be a (not necessarily ergodic) invertible measure-preserving dynamical system and $a_{1},a_{2}$ distinct non-zero integers.
Suppose $f_{1},f_{2}\in L^{\infty}(X)$ with $\norm{f_{i}}_{U^{3}(T)}=0$ for some $i\in\Set{1,2}$.
Then there exists a full measure set $X'\subset X$ such that for every $x\in X'$ the sequence
\[
(f_{1}(T^{a_{1}n}x) f_{2}(T^{a_{2}n}x))_{n}
\]
satisfies the orthogonality criterion \eqref{eq:cond}.
\end{theorem}

Given this result, Theorem~\ref{thm:double-rtt} follows from the known structural theory for $U^{3}$ seminorms.

\begin{proof}[Proof of Theorem~\ref{thm:double-rtt} assuming Theorem~\ref{thm:bilinear-weight-bfko}]
Decompose the functions $f_{1},f_{2}$ according to Theorem~\ref{thm:structure}.
The uniform parts contribute universally good weights for pointwise convergence to zero by Theorems~\ref{thm:bilinear-weight-bfko} and~\ref{thm:bfko}.
The error terms can be controlled by the usual $1$-linear maximal inequality.
It remains to handle the structured parts, and here convergence follows from the nilsequence Wiener--Wintner theorem \cite{MR2544760}.
\end{proof}

The basic idea for verification of the orthogonality criterion is to use Bourgain's bilinear ergodic theorem to convert the lower density in \eqref{eq:cond} to an integral, see Section~\ref{sec:fully-generic}.
The integral is then evaluated using Bourgain's theorem on a certain extension of $X$ that is constructed in Section~\ref{sec:extension}.
A necessary uniformity seminorm estimate is proved in Section~\ref{sec:seminorms}.

\section{Uniformity seminorms}
\label{sec:seminorms}
\begin{definition}
Let $(X,\mu,T)$ be a (not necessarily ergodic) invertible measure-preserving dynamical system and $c$ a non-zero integer.
We define uniformity seminorms by
\[
\norm{ f }_{U^{1}(X,\mu,T,c)} := \norm{ \E(f | I_{T^{c}}) }_{L^{2}},
\]
where $I_{T^{c}}$ is the invariant factor of $T^{c}$, and
\[
\norm{ f }_{U^{l+1}(X,\mu,T,c)}^{2^{l+1}} := \limsup_{H\to\infty} \aveH \norm{ f T^{h} f }_{U^{l}(X,\mu,T,c)}^{2^{l}},
\quad
l\geq 1.
\]
\end{definition}
We omit the parameter $c$ if $c=1$:
\[
\norm{ f }_{U^{l}(X,\mu,T)} := \norm{ f }_{U^{l}(X,\mu,T,1)}.
\]
In this case our definition specializes to the standard (non-ergodic) definition of uniformity seminorms in \cite{MR2795725}.

Using the well-known fact that the limit in the definition of uniformity seminorms exists (even in the uniform Ces\`aro sense) and induction on $l$ one can show that
\begin{equation}
\label{eq:seminorm-ergodic-disintegration}
\norm{ f }_{U^{l}(X,\mu,T,c)}^{2^{l}}
=
\int_{x} \norm{ f }_{U^{l}(X,\mu_{x},T,c)}^{2^{l}}
\end{equation}
for any disintegration $\mu=\int_{x}\mu_{x}$ over a factor contained in the invariant factor $I_{T}$.

We will omit some or all of the subscripts $X,\mu,T$ from the uniformity seminorms when there is no potential for confusion.
We will not verify subadditivity of the functionals $U^{l}(T,c)$ for $c\neq 1$ because it will not be used.

The main structural result about uniformity seminorms in the non-ergodic case is the following.
\begin{theorem}[{\cite[Proposition 3.1]{MR2795725}}]
\label{thm:structure}
Let $(X,\mu,T)$ be a (not necessarily ergodic) measure-preserving dynamical system.
For every $l\geq 0$, every function $f\in L^{\infty}(X,\mu)$ bounded by $1$, and every $\epsilon>0$ there exists a decomposition
\[
f = f_{s} + f_{e} + f_{u}
\]
with $\norm{f_{s}}_{\infty},\norm{f_{e}}_{\infty},\norm{f_{u}}_{\infty}<2$ such that
\begin{enumerate}
\item for almost every $x\in X$ the sequence $(f_{s}(T^{n}x))$ is an $l$-step nilsequence,
\item $\norm{f_{e}}_{L^{1}} < \epsilon$, and
\item $\norm{f_{u}}_{U^{l+1}(X,\mu,T)} = 0$.
\end{enumerate}
\end{theorem}

Note that
\[
c' \mid c
\implies
\norm{ f }_{U^{l}(T,c')} \leq \norm{ f }_{U^{l}(T,c)}
\quad\text{for all } f\in L^{\infty}(X).
\]
The main reason to use the $U^{l}(T,c)$ seminorms instead of the (smaller) $U^{l}(T)$ seminorms is the following estimate, originally proved in \cite[Theorem 12.1]{MR2150389} in the case $a_{i}=i$.
\begin{lemma}
\label{lem:uniformity-seminorm-estimate-for-multilinear-averages}
Let $(X,T)$ be a measure-preserving system, $f_{1},\dots,f_{k}\in L^{\infty}(X)$ functions bounded by $1$, and $a_{1},\dots,a_{k}$ distinct integers.
Then for every F\o{}lner sequence $(\Phi_{N})$ in $\Z$ and every $i$ we have
\[
\limsup_{N\to\infty}
\norm[\big]{ \aveFn T^{a_{1}n}f_{1} \cdots T^{a_{k}n}f_{k} }_{2}
\lesssim_{a_{1},\dots,a_{k},i,c}
\norm{ f_{i} }_{U^{k}(T,c)},
\]
where we can choose $c=a_{1}$ if $k=1$ and $c=a_{i}-a_{i'}$ for any $i'\neq i$ if $k>1$.
\end{lemma}
If one insists on an estimate with $c=1$, then $U^{k+1}(T,1)$ norms have to be used for general $a_{i}$'s, see \cite[Proposition 2]{MR2191208}.
\begin{proof}
By induction on $k$.
If $k=1$, then $c=a_{1}$, and by the mean ergodic theorem and the definition of uniformity seminorms the limit equals
\[
\norm{\E(f_{1} | I_{T^{a_{1}}})}_{L^{2}} = \norm{f_{1}}_{U^{1}(T,c)}.
\]
Suppose that the result is known for some $k$ and consider the case of $k+1$ functions.
Let
\[
u_{n} := T^{a_{1}n}f_{1} \cdots T^{a_{k+1}n}f_{k+1}.
\]
By the van der Corput lemma we have
\begin{equation}
\label{eq:unif-est-vdC}
\limsup_{N\to\infty} \norm{ \aveFn u_{n} }_{L^{2}}^{2}
\lesssim
\liminf_{H\to\infty} \aveH \limsup_{N\to\infty} \abs[\big]{ \aveFn \innerp{u_{n}}{u_{n+h}} }
\end{equation}
Fix distinct indices $i,j$ if $k=1$ or $i,i',j$ if $k>1$.
We have
\begin{align*}
\abs[\big]{ \aveFn \innerp{u_{n}}{u_{n+h}} }
&=
\abs[\big]{ \aveFn \int T^{a_{1}n}(f_{1}T^{a_{1}h}f_{1}) \cdots T^{a_{k+1}n}(f_{k+1}T^{a_{k+1}h}f_{k+1}) }\\
&=
\abs[\big]{ \aveFn \int f_{j}T^{a_{j}h}f_{j} \prod_{i\neq j} T^{(a_{i}-a_{j})n}(f_{i}T^{a_{i}h}f_{i}) }\\
&\leq
\norm[\big]{ \aveFn \prod_{i\neq j} T^{(a_{i}-a_{j})n}(f_{i}T^{a_{i}h}f_{i}) }_{L^{2}}.
\end{align*}
By the inductive hypothesis the $\limsup_{N\to\infty}$ of this is bounded by
\[
\norm{ f_{i}T^{a_{i}h}f_{i} }_{U^{k}(T,c)},
\]
where we can choose $c=a_{i}-a_{j}$ if $k=1$ and $c=(a_{i}-a_{j})-(a_{i'}-a_{j})$ if $k>1$, so that in both cases $c=a_{i}-a_{i'}$ with $i'\neq i$.
It follows that the right-hand side of \eqref{eq:unif-est-vdC} is bounded by
\begin{align*}
\liminf_{H\to\infty} \aveH \norm{ f_{i}T^{a_{i}h}f_{i} }_{U^{k}(T,c)}
&\lesssim_{a_{i}}
\liminf_{H\to\infty} \aveH \norm{ f_{i}T^{h}f_{i} }_{U^{k}(T,c)}\\
\text{by Cauchy--Schwarz}
&\leq
\big( \liminf_{H\to\infty} \aveH \norm{ f_{i}T^{h}f_{i} }_{U^{k}(T,c)}^{2^{k}} \big)^{2^{-k}}\\
&\leq
\norm{ f_{i} }_{U^{k+1}(T,c)}^{2}
\end{align*}
as required.
\end{proof}

The next result is a hybrid between two well-known facts.
Firstly, the nilfactor of step $k$ of a product of two systems is contained in the product of their nilfactors of step $k+1$.
Secondly, the nilfactor of order $k$ of $T^{c}$ is contained in the nilfactor of order $k+1$ of $T$, see \cite[Proposition 2]{MR2191208} and \cite[\textsection 2.2]{MR2795725}.
By proving these two facts simultaneously we lose only one nilpotency step instead of two.

\begin{lemma}
\label{lem:nilfactors-of-product}
Let $(X,T)$ and $(S,Y)$ be measure-preserving systems and $f\in L^{\infty}(X)$, $g\in L^{\infty}(Y)$ be measurable functions.
Then for any non-zero integers $a,b,c$ and $l\geq 1$ we have
\[
\norm{ f \otimes g }_{U^{l}(T^{a}\times S^{b},c)} \leq \abs{ab}^{1/4} \abs{c}^{1/2^{l}} \norm{ f }_{U^{l+1}(T)} \norm{ g }_{U^{l+1}(S)}.
\]
\end{lemma}
\begin{proof}
Consider first the case $l=1$.
By the mean ergodic theorem we have
\begin{align*}
\norm{ f \otimes g }_{U^{1}(T^{a}\times S^{b},c)}^{2}
&=
\lim_{H\to\infty} \aveH \innerp{T^{cah}f \otimes S^{cbh}g}{f\otimes g}\\
&=
\lim_{H\to\infty} \aveH \innerp{T^{cah}f}{f} \innerp{S^{cbh}g}{g}\\
\text{by Cauchy--Schwarz}
&\leq
\limsup_{H\to\infty}
\big( \aveH \abs{\innerp{T^{cah}f}{f}}^{2} \big)^{1/2}\\
&\qquad\cdot
\big( \aveH \abs{\innerp{S^{cbh}g}{g}}^{2} \big)^{1/2}\\
&\leq
\big( \limsup_{H\to\infty} \aveH \norm{ f T^{cah}f }_{U^{1}(T)}^{2} \big)^{1/2}\\
&\qquad\cdot
\big( \limsup_{H\to\infty} \aveH \norm{ g S^{cbh}g }_{U^{1}(S)}^{2} \big)^{1/2}\\
&\leq
\abs{ac}^{1/2} \big( \limsup_{H\to\infty} \aveH \norm{ f T^{h}f }_{U^{1}(T)}^{2} \big)^{1/2}\\
&\qquad\cdot
\abs{bc}^{1/2} \big( \limsup_{H\to\infty} \aveH \norm{ g S^{h}g }_{U^{1}(S)}^{2} \big)^{1/2}\\
&=
\abs{ab}^{1/2}\abs{c} \norm{ f }_{U^{2}(T)}^{2} \norm{ g }_{U^{2}(S)}^{2}
\end{align*}
as required.
Suppose now that the claim is known for some $l\geq 1$, we will show that it holds for $l+1$.
We have
\begin{align*}
\norm{ f \otimes g }_{U^{l+1}(T^{a}\times S^{b})}^{2^{l+1}}
&=
\limsup_{H\to\infty} \aveH \norm{ (f\otimes g) (T^{ah}f \otimes S^{bh}g) }_{U^{l}(T^{a}\times S^{b})}^{2^{l}}\\
\text{by inductive hypothesis}
&\leq
\abs{ab}^{2^{l}/4} \abs{c}
\limsup_{H\to\infty} \aveH \norm{ fT^{ah}f}_{U^{l+1}(T)}^{2^{l}} \norm{ gS^{bh}g }_{U^{l+1}(S)}^{2^{l}}\\
\text{by Cauchy--Schwarz}
&\leq
\abs{ab}^{2^{l-2}} \abs{c}
\big( \limsup_{H\to\infty} \aveH \norm{ fT^{ah}f}_{U^{l+1}(T)}^{2^{l+1}} \big)^{1/2}\\
&\qquad\cdot
\big( \limsup_{H\to\infty} \aveH \norm{ gS^{bh}g}_{U^{l+1}(S)}^{2^{l+1}} \big)^{1/2}\\
&\leq
\abs{ab}^{2^{l-2}+1/2} \abs{c}
\big( \limsup_{H\to\infty} \aveH \norm{ fT^{h}f}_{U^{l+1}(T)}^{2^{l+1}} \big)^{1/2}\\
&\qquad\cdot
\big( \limsup_{H\to\infty} \aveH \norm{ gS^{h}g}_{U^{l+1}(S)}^{2^{l+1}} \big)^{1/2}\\
&=
\abs{ab}^{2^{l-2}+1/2} \abs{c}
\norm{ f }_{U^{l+2}(T)}^{2^{l+2}/2} \norm{ g }_{U^{l+2}(S)}^{2^{l+2}/2}\\
&\leq
\abs{ab}^{2^{l-1}} \abs{c}
\norm{ f }_{U^{l+2}(T)}^{2^{l+1}} \norm{ g }_{U^{l+2}(S)}^{2^{l+1}}
\end{align*}
as required.
\end{proof}

\section{An extension}
\label{sec:extension}
Let $(X,\mu,T)$ be an ergodic measure-preserving dynamical system and let $\pi : (X,T)\to (Z,\alpha)$ be the projection onto the Kronecker factor, where $(Z,\alpha)$ is a compact monothetic group.
Let
\[
\mu = \int_{Z} \mu_{z} \dif z
\]
be the corresponding disintegration, where the integral over $Z$ is taken with respect to the Haar measure.
Fix distinct non-zero integers $a_{1},a_{2}$ and let
\[
\tilde Z := \Set{ (z,z_{1},z_{2}) \in Z^{3} \given (z_{1}-z,z_{2}-z) \in Z'},
\quad
Z' := \overline{\Set{(a_{1}n\alpha,a_{2}n\alpha),n\in\Z}}.
\]
Then $\tilde Z$ is a closed $(\alpha,\alpha,\alpha)$-invariant subgroup of the compact commutative group $Z^{3}$.
Consider the space
\[
\tilde X := \Set{ (x,\xi_{1},\xi_{2}) \in X^{3} \given (\pi x,\pi\xi_{1},\pi\xi_{2}) \in \tilde Z}
\]
with the measure
\[
\tilde\mu := \int_{\tilde Z} \mu_{z_{0}} \otimes \mu_{z_{1}} \otimes \mu_{z_{2}} \dif (z_{0},z_{1},z_{2}),
\]
where the integral is taken with respect to the Haar measure on $\tilde Z$.
With the coordinate projections $\pi_{0},\pi_{1},\pi_{2}$ the space $\tilde X$ becomes a $3$-fold self-joining of $X$ conditionally independent over $\tilde Z$ and invariant under $\tilde T = (T,T,T)$.

\begin{lemma}
\label{lem:lift}
Let $i\in\Set{1,2}$ and $f_{i}\in L^{\infty}(X)$.
Define a function on $\tilde X$ by
\[
F_{i}(x,\xi_{1},\xi_{2}) := f_{i}(x) f_{i}(\xi_{i}).
\]
Then for every $l\geq 1$ and $c\in\N_{>0}$ we have
\[
\norm{F_{i}}_{U^{l}(\tilde X,\tilde\mu,\tilde T,c)} \lesssim_{a_{i},l,c} \norm{f_{i}}_{U^{l+1}(X,\mu,T)}^{2}.
\]
\end{lemma}
\begin{proof}
Since the uniformity seminorm does not change upon passing to an extension, it suffices to estimate the seminorm on the factor $\pi_{0} \vee \pi_{i}$ of $\tilde X$.
This factor is in fact the $2$-fold relatively independent self-joining of $X$ over $I_{T^{a_{i}}}$.
Since $T$ is ergodic on $X$, the invariant factor $I_{T^{a_{i}}}$ is finite, and it follows that the factor $\pi_{0}\vee\pi_{i}$ is isomorphic to a positive measure invariant subset of the product system $X\times X$.
It follows from \eqref{eq:seminorm-ergodic-disintegration} that
\[
\norm{F_{i}}_{U^{l}(\tilde T,c)}
\lesssim_{a_{i}}
\norm{f_{i} \otimes f_{i}}_{U^{l}(T\times T,c)}.
\]
The latter quantity can be estimated by Lemma~\ref{lem:nilfactors-of-product}.
\end{proof}

\begin{corollary}
\label{cor:limit-on-tilde-X-zero}
Let $f_{1},f_{2}\in L^{\infty}(X)$ and suppose $\norm{f_{1}}_{U^{3}(T)}=0$ or $\norm{f_{2}}_{U^{3}(T)}=0$.
Then for $\tilde\mu$-almost every point $(x,\xi_{1},\xi_{2})\in \tilde X$ we have
\[
\lim_{N\to\infty} \frac1N \sum_{n=1}^{N} f_{1}(T^{a_{1}n}x) f_{2}(T^{a_{2}n}x) f_{1}(T^{a_{1}n}\xi_{1}) f_{2}(T^{a_{2}n}\xi_{2}) = 0.
\]
\end{corollary}
\begin{proof}
The left-hand side of the conclusion can be written as
\[
\lim_{N\to\infty} \frac1N \sum_{n=1}^{N} \tilde T^{a_{1}n} F_{1}(x,\xi_{1},\xi_{2}) \tilde T^{a_{2}n} F_{2}(x,\xi_{1},\xi_{2}).
\]
By Bourgain's bilinear pointwise ergodic theorem this limit exists pointwise almost everywhere on $\tilde X$.
On the other hand, by Lemma~\ref{lem:uniformity-seminorm-estimate-for-multilinear-averages} (with $k=2$) and Lemma~\ref{lem:lift} (with $l=2$) the $L^{2}$ limit is zero.
\end{proof}

\section{Fully generic points}
\label{sec:fully-generic}
Recall that a measure-preserving system $(X,\mu,T)$ is called \emph{regular} if $X$ is a compact metric space, $T:X\to X$ is a homeomorphism, and $\mu$ is a Radon probability measure.
\begin{definition}
\label{def:fully-generic}
Let $f_{1},f_{2}\in C(X)$, where $(X,T)$ is a regular ergodic dynamical system.
Let $D_{i}\subset C(X)$, $i=1,2$, be the minimal $T$-invariant sub-$\Q$-algebras containing $f_{i}$.
Fix distinct non-zero integers $a_{1},a_{2}$.
We call a point $x\in X$ \emph{fully generic} for $(f_{1},f_{2})$ if for every function $F \in A(f_{1},f_{2}):=\overline{\lin D_{1}\otimes D_{2}} \subset C(X\times X)$ we have
\begin{equation}
\label{eq:fully-generic}
\lim_{N\to\infty} \aveN F(T^{a_{1}n}x,T^{a_{2}n}x)
=
\int F \dif\nu_{\pi x},
\text{ where }
\nu_{z} = \int_{Z'} \mu_{z + z_{1}} \otimes \mu_{z+z_{2}} \dif(z_{1},z_{2}),
\end{equation}
the latter integral being taken with respect to the Haar measure on $Z'$.
\end{definition}
In other words, $\nu_{\pi x}$ is the natural measure on the set of pairs $(\xi_{1},\xi_{2})$ with
\[
(\pi\xi_{1} - \pi x,\pi\xi_{2} - \pi x) \in Z'.
\]
Note that $A(f_{1},f_{2})$ is a closed sub-$\R$-algebra of $C(X\times X)$.
\begin{lemma}
\label{lem:fully-generic}
For any regular ergodic system $(X,T)$ and any $f_{1},f_{2}\in C(X)$ the set of fully generic points has full measure.
\end{lemma}
\begin{proof}
Recall that the Kronecker factor is characteristic for bilinear ergodic averages in the sense that
\[
\lim_{N\to\infty} \aveN T^{a_{1}n} g_{1} T^{a_{2}n} g_{2}
=
\lim_{N\to\infty} \aveN T^{a_{1}n}\E(g_{1} | Z) T^{a_{2}n}\E(g_{2} | Z)
\]
in $L^{2}(X)$ for any functions $g_{1},g_{2} \in L^{\infty}(X)$.
One way to see this is to show that the limit vanishes if $g_{1} \perp Z$, say.
To this end by Lemma~\ref{lem:uniformity-seminorm-estimate-for-multilinear-averages} it suffices to write $g_{1}$ as a finite linear combination of functions with vanishing $U^{2}(T,c)$ norm, $c=a_{1}-a_{2}$.
Since $T$ is ergodic, the factor $I_{T^{c}}$ consists of finitely many atoms, say $B_{1},\dots,B_{c'}$.
These atoms lie in $Z$, and it follows that $g_{1,j}:=\one_{B_{j}}g_{1} \perp Z$ for every $j$.
On the other hand, any product of $T$-translates of $g_{1,j}$ is supported on an atom of $I_{T^{c}}$, so the $L^{2}$ norm of its projection onto $I_{T^{c}}$ is comparable with the $L^{2}$ norm of its projection onto $I_{T}$.
It follows that $\norm{g_{1,j}}_{U^{2}(T,c)} \approx_{c} \norm{g_{1,j}}_{U^{2}(T)} = 0$.

On the product of Kronecker factors the ergodic averages converge to the integral over the orbit closure, so we obtain
\[
\lim_{N\to\infty} \aveN g_{1}(T^{a_{1}n}x) g_{2}(T^{a_{2}n}x)
=
\int_{(\pi x,\pi x) + Z'} \E(g_{1} | Z) \otimes \E(g_{2} | Z)
=
\int g_{1} \otimes g_{2} \dif\nu_{\pi x}
\]
in $L^{2}$, and by Bourgain's theorem also pointwise almost everywhere.

Hence there is a full measure set $X'$ on which the equality \eqref{eq:fully-generic} holds for functions of the form $F=g_{1}\otimes g_{2}$, $g_{i}\in D_{i}$.
The claim follows by linearity and density.
\end{proof}

We can now proceed with the verification of the orthogonality criterion.

\begin{proof}[Proof of Theorem~\ref{thm:bilinear-weight-bfko}]
The condition \eqref{eq:cond} is measurable (since it suffices to consider rational $\delta$), so in view of \eqref{eq:seminorm-ergodic-disintegration} and by the ergodic decomposition we may assume that $(X,T,\mu)$ is ergodic.
Passing to a suitable topological model we may assume that $(X,\mu,T)$ is a regular ergodic system and $f_{1},f_{2} \in C(X)$.

It suffices to show that
\begin{equation}
\label{eq:thm:bilinear-weight-bfko:claim1}
\int
\lim_{L\to\infty}
\inf_{R\geq L}
\ld(S_{\delta,L,R}((f_{1}(T^{a_{1}n}x) f_{2}(T^{a_{2}n}x))_{n})) \dif\mu(x)= 1
\end{equation}
holds for a fixed $\delta>0$.
Let $\eta_{\delta}$ be a smooth function with $\one_{[-\delta/2,\delta/2]} \leq \eta_{\delta} \leq \one_{[-\delta,\delta]}$.
Then for every $x\in X$ we have
\begin{align*}
\lim_{L\to\infty}
\inf_{R\geq L}&\,
\ld(S_{\delta,L,R}((f_{1}(T^{a_{1}n}x) f_{2}(T^{a_{2}n}x))_{n}))\\
&=
\lim_{L\to\infty}
\inf_{R\geq L}
\liminf_{H\to\infty}\aveHo \prod_{N=L}^{R}\\
&\qquad 1_{\Set{ \abs{ \aveN f_{1}(T^{a_{1}n}x) f_{2}(T^{a_{2}n}x) f_{1}(T^{a_{1}(n+h)}x) f_{2}(T^{a_{2}(n+h)}x) } < \delta }}\\
&\geq
\lim_{L\to\infty}
\inf_{R\geq L}
\liminf_{H\to\infty}\aveHo \prod_{N=L}^{R}\\
&\qquad\eta_{\delta}(\abs[\big]{\aveN f_{1}(T^{a_{1}n}x) f_{2}(T^{a_{2}n}x) f_{1}(T^{a_{1}(n+h)}x) f_{2}(T^{a_{2}(n+h)}x)})
\end{align*}
Now, by the Stone--Weierstrass theorem the function
\[
F_{x,L,R}(\xi_{1},\xi_{2})
:=
\prod_{N=L}^{R} \eta_{\delta}(\abs[\big]{\aveN f_{1}(T^{a_{1}n}x) f_{2}(T^{a_{2}n}x) f_{1}(T^{a_{1}n}\xi_{1}) f_{2}(T^{a_{2}n}\xi_{2})})
\]
lies in $A(f_{1},f_{2})$.
Hence for every $x\in X$ that is fully generic for $(f_{1},f_{2})$ we obtain
\begin{align*}
\lim_{L\to\infty}
\inf_{R\geq L}
\ld(S_{\delta,L,R}((f_{1}(T^{a_{1}n}x) f_{2}(T^{a_{2}n}x))_{n}))
&\geq
\lim_{L\to\infty}
\inf_{R\geq L}
\int F_{x,L,R}(\xi_{1},\xi_{2}) \dif\nu_{\pi x}\\
&\geq
\lim_{L\to\infty}
\int \inf_{R\geq L} F_{x,L,R}(\xi_{1},\xi_{2}) \dif\nu_{\pi x}
\end{align*}

Since the sequence $(\inf_{R\geq L} F_{x,L,R}(\xi_{1},\xi_{2}))$ is monotonically increasing in $L$ for every $x,\xi_{1},\xi_{2}\in X$ and by the monotone convergence theorem we obtain for the left-hand side of \eqref{eq:thm:bilinear-weight-bfko:claim1} the lower bound
\[
\lim_{L\to\infty} \iint
\inf_{R\geq L} F_{x,L,R}(\xi_{1},\xi_{2}) \dif\nu_{\pi x}(\xi_{1},\xi_{2}) \dif\mu(x).
\]
The double integral above is taken with respect to the measure
\[
\int_{X} \delta_{x} \otimes \nu_{\pi x} \dif\mu(x)
=
\int_{Z} \mu_{z} \otimes \nu_{z} \dif z
=
\tilde \mu.
\]
On the other hand, by Corollary~\ref{cor:limit-on-tilde-X-zero} we have
\[
\aveN f_{1}(T^{a_{1}n}x) f_{2}(T^{a_{2}n}x) f_{1}(T^{a_{1}n}\xi_{1}) f_{2}(T^{a_{2}n}\xi_{2})
\to 0
\]
as $N\to\infty$ for $\tilde\mu$-almost every $(x,\xi_{1},\xi_{2})\in \tilde X$.
It follows that
\[
\inf_{R\geq L} F_{x,L,R}(\xi_{1},\xi_{2}) \to 1
\]
as $L\to\infty$ for $\tilde\mu$-almost every $(x,\xi_{1},\xi_{2})\in \tilde X$, and \eqref{eq:thm:bilinear-weight-bfko:claim1} follows by the monotone convergence theorem.
\end{proof}

\printbibliography
\end{document}